\documentclass[a4paper,11pt]{amsart}

\usepackage{amsfonts,amsmath,mathrsfs,amssymb,amsthm,amscd,latexsym,amstext,amsxtra}
 \usepackage[usenames,dvipsnames,svgnames,table]{xcolor}

\usepackage[left=2cm,right=2cm, top=3cm, bottom=3cm]{geometry}

\usepackage[all,cmtip]{xy}
\usepackage{enumerate}
\xyoption{all}
\usepackage{srcltx} 
\usepackage{tikz}
\usepackage{tikz-cd}
\usepackage{textcomp}

\usepackage{amsthm, amsmath, amssymb,latexsym} 
\usepackage{amsfonts,epsfig,amscd}
\usepackage[]{fontenc}
\usepackage{enumerate}
\usepackage[all]{xy}
\usepackage{mathrsfs}

\newtheorem{theorem}{Theorem}[section]
\newtheorem{lemma}[theorem]{Lemma}

\newtheorem{proposition}[theorem]{Proposition}

\newtheorem{definition}[theorem]{Definition}

\newcommand{\nc}{\newcommand}
\newcommand{\cH}{{\mathcal H}}
\newcommand{\cA}{{\mathcal A}}

\newcommand{\cC}{{\mathcal C}}

\newcommand{\cO}{{\mathcal O}}

\newcommand{\cK}{{\mathcal K}}
\newcommand{\cX}{{\mathcal X}}

\newcommand{\cQ}{{\mathcal Q}}

\newcommand{\cU}{{\mathcal U}}

\newcommand{\bH}{{\mathbb H}}

\newcommand{\bG}{{\mathbb G}}
\newcommand{\bC}{{\mathbb C}}

\newcommand{\bZ}{{\mathbb Z}}

\newcommand{\bP}{{\mathbb P}}

\newcommand{\bU}{{\mathbb U}}

\newcommand{\bR}{{\mathbb R}}

\nc{\OLD}{\textcolor{blue}{OLD. (}}
\nc{\NEW}{\textcolor{blue}{NEW. (}}
\nc{\blue}[1]{\textcolor{blue}{ #1}}

\nc{\fr}{{\rightarrow}}
\nc{\co}{{\nabla}}

\nc{\cu}{{\overlineline{\nabla}}}
\nc{\gmc}{\nabla}
\nc{\mtin}[1]{\mbox{{\tiny #1}}}
\nc{\rank}[1]{r_{\mbox{{\tiny #1}}}}

\DeclareMathOperator{\Hom}{Hom}

\DeclareMathOperator{\rk}{rk}

\DeclareMathOperator{\coker}{coker}

\newtheoremstyle{dico}
 {\baselineskip}   
  {\topsep}   
  {}  
  {0pt}       
  {} 
  {.}         
  {5pt plus 1pt minus 1pt} 
  {}          
\theoremstyle{dico}
\newtheorem{say}[theorem]{}
\numberwithin{equation}{section}

\newcommand{\om}{\omega}

\newcommand{\M}{\mathsf{M}}
\newcommand{\Tg}{\mathsf{T}_g}
\newcommand{\Mg}{\mathsf{M}_g}
\newcommand{\A}{\mathsf{A}}
\newcommand{\Ag}{\mathsf{A}_g}

\newcommand{\Cliff}{\operatorname{Cliff}}
\newcommand{\gon}{ \operatorname{gon}}
\newcommand{\sig}{\operatorname{\mathfrak{S}}}
\newcommand{\sigg}{\sig_g}

\newcommand{\gm}{\nabla} 
\newcommand{\nh}{\nabla^{hdg}}
\nc{\lra}{{\longrightarrow}}
\newcommand{\vhs}{\bH_\bZ}
\newcommand{\vhsc}{\bH_\bC}
\newcommand{\bun}{\cH}
\newcommand{\meno}{{-1}}
\newcommand{\codim}{\operatorname{codim}}
\newcommand{\End}{\operatorname{End}}
\newcommand{\id}{\operatorname{id}}
\newcommand{\La}{\Lambda}
\newcommand{\Ann}{\operatorname{Ann}}
\newcommand{\nhom}{\nabla^{\Hom}}
\newcommand{\tj}{\tilde{j}}

\makeatletter 
\@addtoreset{equation}{section}
\makeatother

\begin{document}

\title{Totally geodesic subvarieties in the moduli space of curves}

\author[A. Ghigi]{Alessandro Ghigi}
\address{Dipartimento di Matematica,
	Universit\`a di Pavia,
	Via Ferrata, 5,
	27100 Pavia, Italy}
\email{alessandro.ghigi@unipv.it  }

\author[P. Pirola]{Gian Pietro Pirola}
\address{Dipartimento di Matematica,
	Universit\`a di Pavia,
	Via Ferrata, 5,
	27100 Pavia, Italy}
\email{gianpietro.pirola@unipv.it }

\author[S. Torelli]{Sara Torelli}
\address{Dipartimento di Matematica,
	Universit\`a di Pavia,
	Via Ferrata, 5,
	27100 Pavia, Italy}
\email{sara.torelli7@gmail.com}

\thanks{The authors were supported by PRIN 2015 Moduli spaces and Lie
  Theory, INdAM - GNSAGA,
FAR 2016 (Pavia) ``Variet\`a algebriche, calcolo algebrico, grafi orientati e topologici'' and MIUR, Programma Dipartimenti di Eccellenza
(2018-2022) - Dipartimento di Matematica ``F. Casorati'', Universit\`a degli Studi di Pavia.}
 

 

\date{\today}

\begin{abstract} In this paper we study totally geodesic subvarieties
  $Y \subset \A_g$ of the moduli space of principally polarized
  abelian varieties with respect to the Siegel metric, for $g\geq 4$.
  We prove that if $Y$ is generically contained in the Torelli locus,
  then $\dim Y \leq (7g -2)/3$. 
\end{abstract}

\maketitle

\tableofcontents

\section{Introduction}

Denote by $\M_g$ the moduli space of smooth complex projective curves
of genus $g$, by $\Ag$ the moduli space of principally polarized
abelian varieties and by $j : \M_g \fr \A_g$ the period map.  The
\emph{Torelli locus} $\Tg$ is the closure of $j(\Mg)$ in $\Ag$.  It is
interesting to relate $\Tg$ to the geometry of $\Ag$ as a locally
symmetric variety. We refer to \cite{moonen-oort,CF10,CFG15,deba} for
more information and motivation.  In particular we are interested in
\emph{totally geodesic subvarieties} $Y$ of $\Ag$, i.e. algebraic
subvarieties that are images of totally geodesic submanifolds of
Siegel space $\sigg$.  \emph{Shimura subvarieties} are an important
subclass of totally geodesic subvarieties, related to Hodge theory and
arithmetics \cite{mumford}.  One expects that there are very few
totally geodesic subvarieties of $\Ag$ that are \emph{generically
  contained} in $\Tg$ i.e. such that $Y \subset \Tg$ and
$Y\cap j(\Mg) \neq \emptyset$.  As for Shimura varieties, following
Coleman and Oort, one expects that for large $g$ there are no such
varieties generically contained in $\Tg$, see \cite{moonen-oort}.

An important step in the study of the extrinsic geometry of $\Tg$
inside $\Ag$ was the computation of the second fundamental form of the
period map (which is an embedding outside the hyperelliptic locus
\cite{oort-steenbrink}). This was accomplished in \cite{cpt} and
refined in \cite{CFG15}. Unfortunately this leads only rarely to
explicit formulae. But it is enough to get an upper bound for the
dimension of a totally geodesic subvariety $Y$ generically contained
in $\Tg$ in terms of the gonality of a point of $Y\cap j (\Mg)$, see
\cite{CFG15}. From this one gets a bound without gonality assumptions:
$\dim Y \leq 5 (g-1)/2$, as soon as $g\geq 4$ and $Y$ is not contained
in the hyperellitptic locus.

In this paper we prove the following.
\begin{theorem}\label{Thm-Main2} 
  Let $Y$ be a totally geodesic subvariety of $\A_g$ that is
  generically contained in the Torelli locus.  If $g\geq 4$, then
  $\dim Y \leq (7g -2)/3$.
\end{theorem}

Our proof contains two new ideas. First of all, we use directly the
geodesic curves in $\sigg$. We are able to relate them to Hodge theory
(see Lemma \ref{geodetica}).  The Hodge bundle of the (real
one-dimensional) family of abelian varieties represented by the
geodesic has nice properties with respect to the Fujita decomposition
\cite{Fuj78b,CD:Answer_2017}).  Such properties do not hold for
general families \cite{GT18}.

Secondly, our proof depends heavily on some recent results obtained in
\cite{PT,GST17}.  In fact starting from the geodesics, we build a
complex one-dimensional family of curves (contained in the $Y$) and we
apply the above-mentioned results to this family.  This allows to
split the proof in two cases: in the first case one of the main
theorems in \cite{PT,GST17} yields a map of the whole family onto a
fixed curve. From this one easily gets a better bound on $\dim Y$. In
the second case one is able to control the Clifford index and the
gonality of a point of $Y$. Then an application of the bound in
\cite{CFG15} yields the result.

The plan of the paper is the following: in \S \ref{sec:prel} we recall
some definitions and some results from \cite{PT,GST17}.  In \S
\ref{sec:geo} we study the relation of geodesics to Hodge theory.  In
\S \ref{sec:draft} we prove Theorem \ref{Thm-Main2}.

We refer to \cite{hain,dejong-zhang,gm1,liu-yau-ecc,lu-zuo-HE,toledo}
for related results obtained by different methods.

{\bfseries \noindent{Acknowledgements}}.  
The authors  would like to thank Paola Frediani for interesting discussions.

\section{Preliminaries on weight $1$ variations of the Hodge
  structure.}\label{sec:prel}

The sheaf of $k$-forms (resp.  $(p,q)$-forms) on a complex manifold
$M$ will be denoted by $\cA^k_M$ (resp. $\cA^{p,q}_M$).  The sheaf
holomorphic $k$-forms is denoted by $\Omega^k_M$.

\begin{say}
  \label{secff}
  Given an exact sequence
  $0 \lra E \lra F \stackrel{\pi}{\lra} G \lra 0$ of holomorphic
  vector bundles over a complex manifold $M$ and a (complex)
  connection $\nabla$ on $F$, the second fundamental form
  $\sigma \in \cA^{1,0} _M(E^*\otimes G)$ is defined in the following
  way: given $u\in E_x$, extend $u$ to a local section $\tilde{u}$ of
  $E$ and set $\sigma (u) := \pi( (\nabla \tilde{u}) (x))$.  When
  $\sigma \equiv 0$, we have $\nabla s \in \cA^1_M (F)$ for any
  section $s\in \Gamma (F)$. This means that the connection $\nabla$
  restricts to the bundle $F$ and defines a connection there.

  If $\nabla$ is flat, $\sigma$ is in fact holomorphic, i.e.
  $\sigma \in H^0(B, \Omega^1_B \otimes E^* \otimes G)$.
\end{say}

\begin{say}
  Let $(\vhs, \cH^{1,0}, \cQ)$ be a polarized variation of the Hodge
  structure (shortly, PVHS) of weight $1$ over a complex manifold
  $B$. Here $\vhs$ denotes the local system of lattices, $\cH^{1,0}$
  the Hodge bundle (that in weight $1$ determines the Hodge
  filtration) and $\cQ$ the polarization. We also let
  $\vhsc=\vhs\otimes_\bZ\bC$ denote the local system of complex vector
  spaces and $\bun = \vhsc \otimes _\bC \cO_B$ the associated
  holomorphic flat bundle with flat connection $\gm$. This flat
  holomorphic connection is in fact defined by setting the kernel
  equal to $\bH_\bC$.  The Hodge metric is defined by
  $h(v, w) := i \cQ(v, \bar{w})$.  It is positive definite on
  $\cH^{1,0}$. Hence the orthogonal projection $p: \cH\to \cH^{1,0}$
  is well-defined and
  $\nh=p\nabla_{|\cH^{1,0}}: \cH\to \cH^{1,0}\otimes \cA^1_B$ is the
  Chern connection of the Hermitian bundle $(\cH^{1,0}, h)$.
\end{say}

\begin{say}
  We recall the definition of Siegel upper half-space.  Let
  $\om = \sum_i dx_i \wedge dy_i$ be the standard symplectic form on
  $V:=\bR^{2g}$.  If $J\in \End V$, $J^2 = -\id_V$ and $J^*\om = \om$,
  then $g_J(x,y):= \om (x, J y)$ is a nondegenerate symmetric bilinear
  form on $\bR^{2g}$.  The Siegel upper half-space is defined as
  $\sigg:= \{J\in \End V: J^2= -\id_V, J^*\om=\om, g_J$ is positive
  definite$\} $.  It is a symmetric space of the non-compact type.

\label{tautovhs}
Set $V_\bC:=V\otimes \bC$.  For any $J\in \sigg$ the space $V$ can be
endowed with the complex structure $J$.  We denote by
$V_J^{1,0}\subset V_\bC$ the space of its vectors of type $(1,0)$. We
also set $\cH^{1,0} _J := \Ann (V_J^{0,1}) \subset V_\bC^*$.  Let
$ \La \subset V_\bC^*$ denote the set of forms that are integer-valued
on $\bZ^{2g} \subset V$. $\La$ is a lattice in $V^*$.  The symplectic
form $\om$ induces an isomorphism $\phi: V\cong V^*$ in the usual way.
We denote by $\cQ$ the symplectic form on $V^*$ obtained by
transporting $\om$ to $V^*$ via $\phi$.  With these data we get a
polarized variation of the Hodge structure on $\sigg$: the local
system is $\vhs: = \sigg \times \La$, Hodge bundle is $\cH^{1,0}$ and
the polarization is $\cQ $.
\end{say}

\begin{say}
  If $J \in \sigg$, then $T_J \sigg = \{ X \in \End V: XJ + J X = 0$,
  $\om (X x , y ) + \om ( x , Xy ) =0, \forall x,y\in V\}$.  The
  Siegel upper half-space has an integrable complex structure that on
  $T_J\sigg$ acts by the rule $X \mapsto JX$.  If $X\in T_J\sigg$, we
  can complexify $X$ and its complexification (still denoted by $X$)
  maps $V^{1,0}_J$ to $V_J^{0,1}$ and vice versa.  The transpose of
  $X$, denoted by $X^*$ maps therefore $\cH^{1,0}_J$ to
  $ \cH^{0,1}_J :=\overline{\cH^{1,0}_J}$.  The map
  \begin{gather*}
    X \longmapsto X^*|_{\cH^{1,0}_J } \in \Hom ( \cH^{1,0}, \cH^{0,1})
  \end{gather*}
  yields an isomorphism
  \begin{gather}
    \label{wj}
    T_J\sigg\simeq W_J:=\{L\in \Hom ( \cH^{1,0}, \cH^{0,1}):\cQ
    (L\alpha, \bar{\beta}) + \cQ (\alpha, \overline{L\beta})=0\}.
  \end{gather}
  Therefore we can identify the tangent bundle $T\sigg$ with the
  subbundle $W \subset \Hom ( \cH^{1,0}, \cH^{0,1})$ defined by
  \eqref{wj}.

  The Hodge bundle $\cH^{1,0}$ is provided with the connection $\nh$,
  as happens for every VHS.  Since $\cH^{0,1} \cong (\cH^{1,0})^*$,
  the connection $\nh$ induces a connection on $\cH^{0,1}$ that we
  denote by $\nabla^*$.  So we get an induced connection $\nhom$ on
  the bundle $\Hom ( \cH^{1,0}, \cH^{0,1})$.  The Levi-Civita
  connection of the symmetric metric coincides with the restriction of
  $\nhom$ to $W \cong T\sigg$.
\end{say}

\begin{say}
  Assume now that $B$ is a Riemann surface and that
  $(\vhs, \cH^{1,0}, \cQ) $ is a PVHS of weight 1 on $B$.  Consider
  the exact sequence
  \begin{equation}
    \label{ses-vhs}
    \begin{tikzcd}
      0 \arrow{r} & \cH^{1,0} \arrow{r} & \cH \arrow{r}{\pi^{{0,1}}} &
      \cH/\cH^{1,0} \arrow{r}& 0.
    \end{tikzcd}
  \end{equation}
  Consider on $\cH$ the flat connection and let
  $\sigma: \cH^{1,0}\to \cH/\cH^{1,0}\otimes \Omega_B^1$ be the
  corresponding second fundamental form, as defined in \ref{secff}.
  Using $\gm$ and $\sigma$ we define two vector subbundles of
  $\cH^{1,0}$.
\end{say}

\begin{definition}
  \label{defUK}
  Let $\bU$ denote the subsheaf of $\cH^{1,0}$ spanned by $\nh$-flat
  sections (equivalently by $\gm$-flat sections).  Set \begin{itemize}
  \item[$(i)$] $\cU:=\bU\otimes
    \cO_B,$ 
  \item[$(ii)$]
    $\cK:=\ker ( \sigma : \cH^{1,0} \lra \cH /\cH^{1,0} \otimes
    \Omega_B^1) $.
  \end{itemize}
  We call $\cU$ {\em the unitary flat bundle} and $\cK$ {\em the
    kernel bundle} of the variation, respectively.
\end{definition}
By definition $\cU$ is a holomorphic flat bundle. Since $\nh_{|\cU}$
is the metric connection, it is unitary.  Since $\nabla$ is flat,
$\sigma$ is holomorphic, so $\cK$ is a coherent subsheaf.  If $\sigma$
is a morphism of constant rank, then $\cK$ is also a vector subbundle
of $\cH^{1,0}$.

\begin{proposition}\label{prop-incUK} We
  have $ \cU\subset \cK$.
\end{proposition}
\begin{proof}
  If $u\in \cU_x$, there is a section $s\in \Gamma (A, \bU)$ defined
  on a neighbourhood $A$ of $x$ such that $s(x)=u$. Since
  $\nh s = \gm s \equiv 0$,
  $\sigma (u) = \pi^{0,1} ( (\gm s) (x)) =0$.
\end{proof}

\begin{say}
  The VHS we are interested in come from families of curves.  Let
  $f:\cC\to B$ be a smooth family of genus $g$ curves. This is a
  proper and submersive morphism $f$ from a smooth complex surface
  $\cC$ to a smooth complex curve $B$ whose fibres are curves of genus
  $g$. The map $f$ defines a PVHS $(\vhs, \cH^{1,0},\cQ)$ (called
  geometric) by taking the local system $\vhs:=R^1f_\ast\bZ$ and the
  Hodge bundle
  $\cH^{1,0}:=f_\ast\Omega^1_{\cC/B}=f_\ast\omega_{\cC/B}$.  The
  polarization $\cQ$ is given by the intersection form.  As usual we
  also have the flat bundle $\cH=R^1f_\ast \bC\otimes \cO_B$ with the
  Gauss Manin connection $\gm$, whose flat sections define the local
  system $R^1f_\ast \bC$.
\end{say}

We consider the bundles $\cU$ and $\cK$ of this VHS. In this case
both $\cU$ and $\cK$ can be described in terms of the submersion
$f:\cC\to B$ using holomorphic $1$-forms on $\cC$.  We outline shortly
this description, referring to \cite{PT} and \cite{GST17} for details.
Let
\begin{gather*}
  \Omega^1_{\cC,d}=\ker \{d:\Omega^1_{\cC}\to
  \Omega^2_{\cC}\}\subseteq \Omega^1_{\cC}
\end{gather*}
be the subsheaf of closed holomorphic $1$-forms on $\cC$.

To describe $\cK$ consider the exact sequence
\begin{equation}\label{SeS-DeRhamOnFib}
  \xymatrix@!R{
    {0}  & {f^*\omega_B}  & {\Omega^1_\cC}  & {\Omega^1_{\cC/B}\simeq \omega_{\cC/B}}  & {0}                                  & 
    \ar"1,1";"1,2"\ar"1,2";"1,3"\ar"1,3";"1,4"\ar"1,4";"1,5"
    \hole
  }
\end{equation}
defined by duality using the morphism $d f:T_\cC\to f^*T_B$. Here the
cokernel $\Omega^1_{\cC/B}$ is the sheaf of relative differentials and
$\Omega^1_{\cC/B}\simeq \omega_{\cC/B}$ since $f$ is smooth.  Pushing
forward \eqref{SeS-DeRhamOnFib} to $B$, we get the exact sequence
\begin{equation*}
  \xymatrix@!R{
    {0}  & {f_*f^*\omega_B\simeq\omega_B}  & {f_*\Omega^1_{\cC}}  & {f_*\om_{\cC/B}}  & {(R^1f_*\cO_{\cC})\otimes\omega_B} . 
    \ar"1,1";"1,2"\ar"1,2";"1,3"\ar"1,3";"1,4"\ar"1,4";"1,5"^-{\partial}
    \hole
  }
\end{equation*}
By a fundamental result of Griffiths (see
\cite{Grif_OnthePeriodsI&II_1969} and \cite[Ch. 10]
{V_HodgeTheoryI_2002}) $\partial$ is the vector bundle morphism that
acts on the fibre over $b\in B$ as follows:
\begin{gather}
  \label{griffini}
  \partial_b : H^0(C_b , \om_{C_b}) \lra H^1(C_b, \cO_{C_b})\otimes
  T_b^*B , \quad \partial_b (\om): = \xi_b \cup \om,
\end{gather}
where $\xi_b : T_bB \fr H^1(C_b,T_{C_b})$ is the Kodaira-Spencer map.
Using the isomorphism
$\cH_b/\cH^{1,0}_b \cong H^{0,1}(C_b) \cong H^1(C_b, \cO_{C_b})$ one
can identify $\partial_b$ with the second fundamental form $\sigma_b$
of \eqref{ses-vhs}.  Consequently, we have that
\begin{gather}
  \label{Kpartial}
  \cK= \ker \partial.
\end{gather}
Since $\dim B=1$, if $v\in T_bB$ is a non-zero vector, then we have
\begin{gather*}
  \cK_b = \ker \sigma_b(v) = \ker \xi_b (v).
\end{gather*}
In particular $\sigma$ has constant rank and $\cK$ is a vector
subbundle of $f_*\om_{\cC/B}$.  Moreover the sequence
\begin{equation*}
  \begin{tikzcd}
    0 \arrow{r} & \omega_B \arrow{r} & f_*\Omega^1_{\cC} \arrow
    {r}{\pi} & \cK \arrow{r} & 0.
  \end{tikzcd}
\end{equation*}
is exact.

We recall shortly the definition of Massey products.
Fix $b\in B$ and a generator $v\in T_b B$.  Using $v$ we get an
isomorphism $T_bB \cong \bC$. So the exact sequence \eqref
{SeS-DeRhamOnFib} restricted to the fibre  $C_b$ reads
\begin{equation}
  \label{SeS-DeRhamOnSFibre}
  \begin{tikzcd}
    0 \arrow{r} & \cO_{C_b} \arrow{r} & \Omega^1_{\cC}|_{C_b}
    \arrow{r}& \om _{C_b} \arrow{r} & 0.
  \end{tikzcd}
\end{equation}
We set for simplicity
\begin{gather}
  E_b:= \Omega^1_{\cC}|_{C_b}.
\end{gather}
The extension class of \eqref {SeS-DeRhamOnSFibre} is
$ \xi_b(v) \in H^1(C_b, T_{C_b})$. It follows that
$\det E_b \cong \omega_{C_b}$. So we get the \emph{adjoint map}, first
defined in \cite{C-P_TheGriffiths_1995}:
\begin{equation}\label{Mor-Mp/Aj}
  \Phi_b\colon \xymatrix@!R{
    {\bigwedge^2H^0(E_b )}  & {H^0(\bigwedge^2E_b)\simeq H^0(\omega_{C_b})}.                  
    \ar"1,1";"1,2"
    \hole
  }
\end{equation}
The long cohomology exact sequence associated to
\eqref{SeS-DeRhamOnSFibre} starts as follows:
\begin{equation*}
  \begin{tikzcd}
    0 \arrow {r} & \bC \arrow {r} & H^0(C_b , E_b) \arrow {r}{p} &
    H^0(C_b ,\om_{C_b}) \arrow {r}{\cup\xi_b(v)} & H^1(C_b ,\cO_{C_b})
  \end{tikzcd}
\end{equation*}
Given $v_1,v_2$ in $\cK_b = \ker \cup \xi_b$, we lift them, i.e. we
take vectors $\tilde{v_1}, \tilde{v_2} \in H^0(C_b, E_b)$ such that
$p (\tilde{v}_i )= v_i$.  Let
$ \langle v_1,v_2\rangle \subset H^0(C_b, \om _{C_b})$ denote the span
of $v_1, v_2$.
\begin{definition}\label{def:MP}
  The element
  \begin{gather*}
    m_b(v_1,v_2)= [\Phi_b(\tilde{v_1}, \tilde{v_2})]\in
    H^0(\omega_{C_b})/\langle v_1,v_2 \rangle
  \end{gather*}
  is independent of the choice of the liftings and it is called \emph
  {Massey product} of $v_1,v_2\in \ker \cup \xi_b$.
\end{definition}
We notice that $m_b(v_1,v_2)= [\Phi_b(\tilde{v_1}, \tilde{v_2})] =0$
if and only if
$\Phi_b(\tilde{v_1}, \tilde{v_2})\in \langle v_1,v_2 \rangle$.  Since
we are assuming that $f$ is submersive, we have
$\cK_b=\ker\cup_{\xi_b}$ by \eqref{Kpartial}, so the Massey product is
defined for any $v_1, v_2 \in \cK_b$ and for any $b\in B$.

We now consider the bundle $\cU$ for the VHS coming from the family
$f:\cC \fr B$.  On $\cC$ there is an exact sequence of sheaves
\begin{equation*}
  \begin{tikzcd}
    {0} \arrow{r} & {\bC_{\cC}} \arrow{r} & {\cO_{\cC}} \arrow{r} {d}
    & {\Omega^1_{\cC,d}} \arrow{r} & {0.}
  \end{tikzcd}
\end{equation*}
We push it forward to $B$.  Since $f$ is a submersion with compact
connected fibres, $f_\ast\bC\simeq \bC_{B}$,
$f_\ast \cO_{\cC}\simeq \cO_B$ and
$\omega_B= \coker(d:\bC\to \cO_b)=\coker (f_\ast d:f_\ast\bC\to f_\ast
\cO_{\cC} )$. So we get the exact sequence
\begin{equation*}
  \xymatrix@!R{
    {0}  & {\omega_B}  & {f_*\Omega^1_{\cC,d}}  & {R^1f_\ast \bC}  & {R^1f_\ast \cO_{\cC}.}& 
    \ar"1,1";"1,2"\ar"1,2";"1,3"\ar"1,3";"1,4"\ar"1,4";"1,5"
    \hole
  }
\end{equation*}
By the Splitting Lemma \cite[Lemma 3.2]{PT} the local system $\bU$
underlying $\cU$ fits into the above exact sequence as
\begin{equation*}
  \xymatrix@!R{
    {0}  & {\omega_B}  & {f_*\Omega^1_{\cC,d}}  & {\bU}  & {0.} 
    \ar"1,1";"1,2"\ar"1,2";"1,3"\ar"1,3";"1,4"\ar"1,4";"1,5"
    \hole
  }
\end{equation*}

%
%
%
%

Since $\cU\subset \cK$, we can restrict our attention to Massey
products on $\cU$.
\begin{proposition}\label{prop-liftMT} Assume that the Massey products
  of $\cU$ are all zero, i.e. $m(v_1, m_2) =0$ for any
  $v_1 , v_2 \in \cU_b$ and for any $b\in B$.  Then, for any local
  flat frame $s_1, \dots, s_k$ of $\cU$, there are
  $\omega_1, \dots , \omega_m \in H^0(B,f_\ast\Omega^1_{\cC,d})$ such
  that $\pi (\om_j ) = s_j$ and $\omega_i\wedge \omega_j=0$ for any
  $i$ and $j$.
\end{proposition}
See \cite[Prop. 4.3]{PT} for the proof.  Massey products on $\cU$
contain deep geometric information as is shown by the following \emph
{tubular} version of the classical Castelnuovo de-Franchis theorem.

\begin{theorem}[\cite{GST17}] \label{thm:tubularCdF} Let
  $f: S \to \Delta$ be a submersive family of smooth projective curves
  over a disk $\Delta$. Let
  $\omega_1,\ldots,\omega_k \in H^0\left(S,\Omega_S^1\right)$
  ($k \geq 2$) be closed holomorphic 1-forms such that
  $\omega_i \wedge \omega_j = 0$ for every $i,j$, and whose
  restrictions to a general fibre $F$ are linearly independent. Then
  (possibly after shrinking $\Delta$) there exist a projective curve
  $C$ and a morphism $\phi: S \to C$ such that
  $\omega_i \in \phi^*H^0\left(C,\omega_C\right)$ for every $i$.
\end{theorem}

\section{A lemma on geodesics}\label{sec:geo}

\begin{say}
  Let $B \subset \sig_g$ be a complex submanifold with $\dim_\bC B =1$
  and consider the restriction to $B$ of the tautological PVHS on
  $\sigg$ introduced in \ref{tautovhs}.  Consider the exact sequence
  \begin{equation*}
    \begin{tikzcd}
      0 \arrow{r} & \cK \arrow{r} & \cH 
      \arrow{r}{\pi} & \cH/\cK \arrow{r} & 0.
    \end{tikzcd}
  \end{equation*}
  Using on $\cH$ the flat connection $\nabla$ we get a second
  fundamental form as described in \ref{secff}:
  \begin{gather*}
    \tau_x : T_xB \otimes \cK_x \lra \cH_x/\cK_x, \quad
    \tau_x(v\otimes e) := \pi ( (\nabla_v \tilde{e}) (x)),
  \end{gather*}
  where $\tilde{e}$ is a local section of $\cK$ such that
  $\tilde{e}(x) = e$.  Moreover $\tau$ is a holomorphic section of
  $\Omega^1_B\otimes \cK ^* \otimes \cH / \cK$.
\end{say}

\begin{lemma}
  If $\tau \equiv 0$ on $B$, then $\cK = \cU$.
\end{lemma}
\begin{proof}
  If $\tau \equiv 0$, the connection $\nabla$ preserves the subbundle
  $\cK$.  If $u\in \cK_x$, let $s$ be a local section of $\cH$ such
  that $s(x)=u$ and $\gm s=0$. Then $s$ lies in $\cK$ since $\cK$ is
  preserved by the parallel displacement of $\nabla$. Thus $s$ in fact
  lies in $\bU$ and $u\in \cU_x$. Hence $\cK \subset \cU$.  The
  opposite inclusion is always true.
\end{proof}

If $\gamma : \bR \fr \sigg$ is a non-constant geodesic and
$-\infty < a < b < +\infty$, we call $\Gamma:=\gamma([a,b])$ a
\emph{geodesic segment}.

\begin{lemma}
  \label{geodetica}
  Let $B \subset \sig_g$ be a complex submanifold with
  $\dim_\bC B =1$.  If $B$ contains a geodesic segment, then
  $\cK=\cU$.
\end{lemma}
\begin{proof}

  By the previous lemma it is enough to show that $\tau \equiv 0$ on
  $B$.  Since $\tau$ is holomorphic, if we prove that $\tau$ vanishes
  on $\Gamma$, then $\tau \equiv 0 $ on all $B$ by the identity
  principle.

  To show that $\tau=0$ on $\Gamma$ fix $x=\gamma(t_0) \in \Gamma$,
  $v \in T_xB$ and $e \in \cK_x$.  Since $T_xB$ has complex dimension
  1, $v = \lambda \dot{\gamma}(t_0)$ for some $\lambda \in \bC$, so it
  is enough to consider $v= \dot{\gamma}(t_0)$.  Let
  $\tilde{e}=\tilde{e}(t)$ be the section of $\cK$ over $\Gamma$
  obtained by parallel translation of the vector $e$ with respect to
  the connection $\nh$. A priori $\tilde{e}$ is only a section of
  $\cH^{1,0}$.  We claim that in fact
  $\tilde{e}(t) \in \cK_{\gamma(t)}$.

  Set $\cH^{0,1}:=\cH / \cH^{1,0} $.  Since
  $\cH^{0,1} \cong (\cH^{1,0})^*$, the connection $\nh$ induces a
  connection on $\cH^{0,1}$ that we denote by $\nabla^*$.  We get an
  induced connection on $\Hom (\cH^{1,0}, \cH^{0,1}) $ that we denote
  by $\nhom$.  The tangent bundle $T\sigg$ is a subbundle of
  $\Hom (\cH^{1,0}, \cH^{0,1}) $ and the Levi-Civita connection for
  the symmetric metric agrees with the restriction of the connection
  on $\Hom (\cH^{1,0}, \cH^{0,1}) $.  So if $\xi=\xi(t) $ is a section
  of $T\sigg$ and $s=s(t)$ is a section of $\cH^{1,0}$ we have
  \begin{gather*}
    \nabla^* _{\dot{\gamma}} (\xi (s)) = (\nhom _{\dot{\gamma}} \xi)
    (s) + \xi (\nh _{\dot{\gamma}} s).
  \end{gather*}
  Take $\xi = \dot{\gamma}$ and $s=\tilde{e}$. We have
  \begin{gather*}
    \nhom _{\dot{\gamma}} \dot{\gamma} =
    \nabla^{\operatorname{Levi-Civita}} _{\dot{\gamma}} \dot{\gamma} =
    0, \qquad \nh _{\dot{\gamma}} \tilde {e} = 0.
  \end{gather*}
  So $\nabla^*_{\dot{\gamma}} ( \dot{\gamma}(\tilde{e})) \equiv 0$.
  Recall now that $e = \tilde{e}(t_0) \in \cK_x$.  Moreover
  $\cK_{\gamma(t)} = \{ u \in \cH^{1,0}_x : \dot{\gamma}(t)(u) = 0\}$.
  So $\dot{\gamma} (\tilde{e}) (t_0) =0$.  Since we have checked that
  $\dot{\gamma}(\tilde{e})$ is a parallel section over $\Gamma$, we
  conclude that $\dot{\gamma} (\tilde{e}) \equiv 0$ on $\Gamma$. This
  means that $\tilde{e}(t) \in \cK_{\gamma(t)}$ for any $t$, as
  claimed.

  To conclude the proof, notice that on $\cK$ the connections $\nh$ e
  $\nabla$ coincide.  So
  \begin{gather*}
    \nabla_{\dot{\gamma}(t)} \tilde{e} = \nh_{\dot{\gamma}(t)}
    \tilde{e} \equiv 0.
  \end{gather*}
  Since $\tilde{e}$ is a section of $\cK$ we can use it to compute the
  second fundamental form and we get
  \begin{gather*}
    \tau_x(v, e) = \pi (\nabla _{\dot{\gamma}(t)} \tilde{e} ) (t_0) =0
    .
  \end{gather*}
\end{proof}

\begin{lemma}
  \label{analitica}
  Let $f : I=[a,b] \fr M$ be a real analytic map in a complex manifold
  $M$.  Then there is an open subset $A\subset \bC$ containing $I$ and
  a holomorphic extension $h : A \fr M$ of $f$.
\end{lemma}
\begin{proof}
  We can easily find a finite family of disks $\{D_i \}_{i=1}^m$
  centred at points $t_i \in I$ such that (a) $f(D_i \cap I)$ is
  contained in the domain of a chart $U_i$
  $(U_i,\phi_i=( z^1_i , \ldots, z_i^n))$ of $M$, (b) on $D_i$ there
  are $n$ holomorphic functions $h_i^j$, $j=1, \ldots, n$, (c)
  $z^j_i\circ f =h_i^j$ on $D_i\cap I$.  The function
  $f_i:=\phi_i\meno \circ (h_i^1, \ldots, h_i^n) $ is a holomorphic
  extension of $f|_{ D_i\cap I}$ to $D_i$.  We claim that these
  functions glue together.  Indeed we start by setting $h_1:=f_1$ on
  $D_1$ and we proceed inductively. Assume that a holomorphic
  extension $h_{k-1}$ is given on $D_1 \cup \cdots \cup D_{k-1}$. We
  claim that $h_{k-1} = f_k$ on
  $D_k \cap (D_1 \cup \cdots \cup D_{k-1})$.  Indeed this set is
  connected and contains a subinterval of $I$.  $h_{k-1} = f = f_k$ on
  this subinterval.  Therefore by the identity principle $h_{k-1}=f_k$
  on $D_k \cap (D_1 \cup \cdots \cup D_{k-1})$. Thus using $h_{k-1}$
  and $f_k$ we get a well-defined holomorphic extension $h_k$ on
  $D_1 \cup \cdots \cup D_k$. At the end it is enough to set $h:=h_n$.
\end{proof}

\section{Proof of the Theorem}
\label{sec:draft}

\begin{proposition}
  \label{secanti}
  Assume that $g\geq 3$ and let $Y \subset \M_g $ be a subvariety of
  of codimension $c$.  For a smooth point $y\in Y$, there is
  $\xi\in T_yY$ such that
  \begin{gather}
    \label{stima-bis}
    \dim (\ker \cup\xi) \geq g - k_0 -1,
  \end{gather}
  where
  \begin{gather}
    \label{k0}
    k_0:= \Bigl \lceil \frac{c-1}{2} \Bigr \rceil.
  \end{gather}
\end{proposition}
\begin{proof}
  Restricting $Y$ we can assume that it embeds in the Kuranishi
  family. So $T_y Y \hookrightarrow H^1(C,T_C)$ where $[C]=y$.
  Consider the bicanonical image
  $X:=\operatorname {Bic} (C) :=\phi_{|2K_C|} (C) \subset
  \bP(H^1(C,T_C)) =\bP H^0(C, \omega_C^{\otimes 2})^{\vee}$.  Denote
  by $S^kX$ the variety of $k$-secants of $X$.
  Since $ \dim S^kX = 2k+1 $, we have
  $\dim S^{k_0} X \geq c =\codim (\bP(T_yY) \subset \bP(H^1(C,T_C)))$.
  Hence $S^{k_0} \cap \bP(T_yY) $ contains at least some point
  $[\xi]$.  By construction there is an effective divisor $D$ of
  degree $k_0+1$ such that
  \begin{gather}
    \label{suppo}
    \xi \in \ker ( \rho_D : H^1(C,T_C) \lra H^1(C,T_C(D))
  \end{gather}
  where $\rho_D$ is the map induced by the inclusion
  $T_C \hookrightarrow T_C(D)$.  Indeed, given an effective divisor
  $D$, denote by $\langle D\rangle$ the intersection of all
  hyperplanes $H \subset \bP H^0(C,\omega^{\otimes 2})^\vee$ such that
  $D\leq \phi_{|2K_C|}^*H$.  Set
  $X_{k_0+1}:= \{(D,p) \in C^{(k_0+1)}\times \bP H^0(C,\omega^{\otimes
    2})^\vee: p\in \langle D\rangle \}$ and denote by $p_2$ the second
  projection. Then $S^{k_0}X=p_2(X_{k_0+1})$. So
  $[\xi] \in \langle D\rangle $ for some $D\in C^{(k_0+1)}$, which
  yields \eqref{suppo}.  But
  \begin{gather*}
    \dim \ker ( \cup \xi: H^0(C,\om_C) \fr H^1(C, \cO_C) ) \geq
    g -\deg D = g - k_0 -1.
  \end{gather*}
  (See e.g. \cite[Lemma 2.3]{BGN_Xiao_2015}.)
\end{proof}

\begin{proof}[Proof of Theorem \ref{Thm-Main2}]
  We argue by contradiction, assuming the existence of a totally
  geodesic subvariety $Y\subset \A_g$ that is generically contained in
  $\M_g$ and with $\dim Y > (7g-2)/3$. If $c$ denotes the codimension
  of $Y\cap \M_g$ in $\M_g$, this is equivalent to
  \begin{gather}
    \label{assurdo}
    c < \frac{2g - 7}{3} .
  \end{gather}
  Observe that for $k_0$ defined in \eqref{k0} this implies
    \begin{gather}
      \label{eq:100}
      2k_0\leq g-4.
    \end{gather}
    Observe also that by dimension $Y$ is not contained in the
    hyperelliptic locus. Fix a smooth point $y\in Y$ that
    represents a non-hyperelliptic curve. By Proposition
  \ref{secanti} there is $\xi\in T_yY$ such that $\xi\neq 0$ and
  \begin{gather}
    \label{rangoxi}
    \dim (\ker \cup\xi) \geq g - k_0 -1,
  \end{gather}
  where $k_0$ is defined as in
  \eqref{k0}.  Let $\Delta$ be a polydisk and let $F: \cX \fr \Delta$
  be a Kuranishi family with $[X_0]=y$.  The moduli map
  $\pi: \Delta \fr \M_g$, $\pi(t): = [X_t]$ is finite, satisfies
  $\pi(0) = y$ and its image is a neighbourhood of $y$.  The period
  mapping $j$ can be lifted to a map $\tj: \Delta \fr \sigg$.
  \begin{equation*}
    \begin{tikzcd}
      \Delta \arrow{r}{\tj} \arrow{d}{\pi} & \sigg \arrow{d} \\
      \M_g \arrow{r}{j} & \A_g.
    \end{tikzcd}
  \end{equation*}
  Thus $Y':=\tj ( \pi^{-1}( Y)) $ is a germ of totally geodesic
  submanifold of $\sigg$, that contains the point $y':=\tj(0)$ and is
  contained in $\tj(\Delta)$.  Let $\gamma: \bR \fr \sigg$ be the
  geodesic in $\sigg$ such that $\gamma(0)=y'$ and
  $\dot{\gamma}(0) = d\tj(\xi)$.  Since $\xi\neq 0$ the curve $\gamma$
  is non-constant.  Moreover it is real analytic since the Siegel
  metric is real analytic.  Fix $\varepsilon>0$ such that
  $\gamma ([-\varepsilon,\varepsilon]) \subset Y'$.  By Lemma
  \ref{analitica} there is an open subset $A \subset \bC$ containing
  $[-\varepsilon,\varepsilon]$ and a holomorphic extension
  $h : A \fr Y'$ of $\gamma$.  Restricting $A$ we can assume that $h$
  is an embedding and that it avoids the hyperelliptic locus. Set
  $B:=h(A)$.  Since $B\subset Y' \subset \tj(\Delta)$ we can restrict
  the Kuranishi family to $\tj^{-1}(B)\simeq B$ and we get a universal
  family of curves $f: \cC \fr B$ such that $b_0:=h(0) = \tj(y)$.  The
  VHS of this family is simply the restriction to $B$ of the
  tautological VHS on $\sigg$ described in \ref{tautovhs}.  We
  consider the bundles $\cU$ and $\cK$ for this VHS on $B$ (see
  Definition \ref{defUK}).  By Lemma \ref {geodetica} we have
  $\cU=\cK$.  But using \eqref{griffini} we see that
  $\cK_{b_0}= \ker ( \cup \dot{\gamma}(0)) = \ker (\cup \xi) $. At
  this point we use the bound \eqref{rangoxi}.  Summing up $\cK=\cU$
  has rank at least $ g - k_0 -1$.

  Now we consider the Massey products of $\cK=\cU$ on $B$ (see
  definition \ref{def:MP}).

  Assume first that these are identically 0.  Take a basis
  $u_1, \dots , u_m$ of $\cK_b$. Up to shrinking $B$ we can extend
  these vectors to flat sections $u_1, \dots, u_m$ of $\cU$ on $B$. By
  Proposition \ref{prop-liftMT}, there exist unique liftings to
  sections $\omega_1,\ldots,\omega_{m}$ of $f_\ast\Omega^1_{\cC, d}$
  (i.e. of closed holomorphic $1$-forms on $\cC$) such that
  $\omega_i\wedge \omega_j=0$, for any $i,j$.  By Theorem
  \ref{thm:tubularCdF} (i.e. \cite[Theorem 1.5]{GST17}), we get a a
  morphism $\phi:\cC\fr C'$ onto a genus $g'\geq 2$ smooth compact
  curve $C'$, whose restriction to every fibre of $f$ gives a
  non-constant degree $n$ morphism $ \phi:C_b\fr C'$ such that
  $\omega_1,\ldots,\omega_{m}\in \phi^\ast H^0(\omega_{C'})$. It
  follows immediately that $\rk \cU\leq g'$. But also
  $g'\leq \rk \cU$, since any section given by pull back from $C'$ is
  flat and has wedge zero with the others. So we conclude that
  $\rk \cU= g'$.  Recalling the bound on $\rk \cK$ established above,
  we get
  \begin{gather*}
    g(C')= \rk \cU = \rk \cK \geq g - k_0 -1.
  \end{gather*}
  Since $f$ is non isotrivial by construction, $n\geq 2$, so by the
  Riemann-Hurwitz formula
  \begin{gather*}
    2g-2 \geq 4g -4k_0 -4 -4 , \quad 2 k_0 \geq g -3 .
  \end{gather*}
  But we were assuming \eqref{assurdo} and hence \eqref{eq:100}. So we
  would get $ g-3 \leq 2k_0 \leq g-4$, which is clearly absurd. This
  shows that the Massey products cannot vanish identically.

  So there is $b\in B$ and $u_1, u_2 \in \cK_b$, with
  $m(u_1, u_2) \neq 0$. In particular there are
  $\tilde{u}_i \in H^0(E_b)$ with
  $\tilde{u}_1\wedge \tilde{u}_2 \neq 0$.  Before using this
  information, we need to recall a construction already used e.g. in
  \cite [p. 428-429]{MNP16}.  Consider the sequence
  \eqref{SeS-DeRhamOnSFibre} with extension class $\xi:=\xi_b(v)$.
  From the associated cohomology sequence
  \begin{gather*}
    0 \fr H^0(\cO_{C_b}) \cong \bC \fr H^0(E_b) \fr H^0 (\om_{C_b})
    \stackrel{\cup \xi} {\lra} H^1 (\cO_{C_b})
  \end{gather*}
  we get $h^0 (E_b) = \dim \ker (\cup \xi) + 1$.  Recall that
  $\det E_b \cong \om_{C_b}$ and the definition of the adjunction map
  \eqref{Mor-Mp/Aj}.
  We claim that $\ker \Phi_b$ contains a non-zero decomposable
  element.  Indeed we have
  $ h^0(E_b)= \dim \ker (\cup \xi) + 1 \geq g - k_0 $, hence
  $ 2h^0(E_b) -4 \geq 2g -2k_0 -4$.  But recall that we are assuming
  \eqref{assurdo}. Hence from \eqref{eq:100}
  we get
  \begin{gather}
    \label{indecompo}
    2h^0(E_b) -4 \geq g=h^0(\om_{C_b}).
  \end{gather}
  Since
  $g \geq \dim \Lambda^2 H^0(E_b) - \dim \ker \phi = \dim \bP
  (\Lambda^2 H^0(E_b)) - \dim \bP(\ker \phi)$ and
  $\dim \bG (2, \Lambda^2H^0(E_b)) = 2 h^0(E_b) -4$, it follows from
  \eqref{indecompo} that
  $ \bG (2, \Lambda^2H^0(E_b)) \cap\bP(\ker \phi) \neq \oslash$, so
  there is a decomposable element
  $s_1\wedge s_2 \in \ker \phi - \{0\}$, as claimed.  In other words
  there are linearly independent sections $s_1, s_2 \in H^0(E)$ such
  that $s_1(x)$ and $s_2(x)$ are always proportional.  Let
  $\mathfrak{F}$ be the subsheaf of $\cO_{C_b}(E_b)$ generated by
  $s_1$ and $s_2$. The saturation $\mathfrak{L}$ of $\mathfrak{F}$ is
  the sheaf of sections of a line bundle $L$.  By construction
  $s_1,s_2 \in H^0(L)$ are linearly independent, so $h^0(L) \geq 2$.
  The quotient sheaf $\cO_{C_b}(E_b) / \mathfrak{L}$ is also the sheaf
  of sections of a line bundle $M$. Thus
  \begin{equation*}
    \begin{tikzcd}
      & & 0 \arrow{d} & & \\
      & & L \arrow{d}{\alpha} & & \\
      0 \arrow {r} & \cO_{C_b} \arrow{r}{i} \arrow{rd}{s_3} & E_b
      \arrow{d}{\beta} \arrow{r} &
      \om_{C_b} \arrow{r} & 0 \\
      & & M \arrow{d} & &\\
      && 0 &&
    \end{tikzcd}
  \end{equation*}
  $s_3 :=\beta \circ i$ is a section of $M$. We claim that
  $s_3 \not \equiv 0$. Indeed if $s_3 \equiv 0$, we would have
  $i(1) = \alpha (\sigma)$ for a nonvanishing section
  $\sigma \in H^0(L)$. But then $L$ would be trivial and $s_1$ and
  $s_2$ would be linearly dependent.  Thus $s_3 \not\equiv 0$. Let $D$
  be the divisor of zeros of $s_3$. Then $M=\cO_{C_b}(D)$.  Since
  $\det E_b = L \otimes M \cong \om_{C_b}$, $L = \om_{C_b}(-D)$.

  Now we are finally able to use $\tilde{u_1}$ and $\tilde{u_2}$.
  Since $\tilde{u}_1\wedge \tilde{u}_2 \neq 0$, the sections
  $\tilde{u}_i$ do not lie both in $L$. Hence at least one of them has
  a non-trivial image $s_4$ in $M$. The section $s_3$ and $s_4$ are
  independent, so $h^0(M) = h^0 (\cO(D)) \geq 2$.  From the diagram
  one gets
  \begin{gather*}
    \dim \ker (\cup \xi )\leq g - (\deg D -2h^0(D) +2 ) = g-\Cliff(D).
  \end{gather*}
  (see e.g. \cite[Lemma 2.3]{BGN_Xiao_2015}).  Since $h^0(D) \geq 2$
  and $h^0(\om _C(-D)) \geq 2$, the divisor $D $ contributes to the
  Clifford index.  Therefore $\Cliff(C_b) \leq \Cliff (D)$. It is
  known that $ \operatorname{gon} (C_b) \leq \Cliff (C_b) +3$, see
  \cite[Thm .2.3]{MC91}.  By the Lemma and \ref{stima-bis}
  $ \Cliff(D) \leq g - \dim \ker(\cup \xi) \leq k_0 + 1 $.  Thus
  \begin{gather*}
    \gon (C_b) \leq \Cliff (C_b) +3 \leq \Cliff(D) + 3 \leq k_0 + 4.
  \end{gather*}

  Now we can apply \cite[Theorem 4.2]{CFG15}: since $Y \subset \M_g$
  is totally geodesic and $[C_b] \in Y$ is not hyperelliptic, we have
  \begin{gather*}
    \dim Y \leq 2g +\gon (C_b) - 4.  
  \end{gather*}
  For $g\geq 4 $ we have $ g-2 \geq (2/3)g - 1 $.  Hence
  \begin{gather*}
    3g -3 -c =\dim Y \leq 2g + k_0 + 4 - 4= 2g +\Bigl \lceil
    \frac{c-1}{2} \Bigr \rceil \leq 2g + \frac{c-1}{2} +1.
  \end{gather*}
  But this gives $ (2g - 7) /3 \leq c$, which yields the desired
  contradiction with \eqref{assurdo}.
\end{proof}


\begin{thebibliography} {99}

\bibitem{BGN_Xiao_2015} M.~A. Barja, V.~Gonz\'alez-Alonso, and
  J.~C. Naranjo.  \newblock Xiao's conjecture for general fibred
  surfaces (arxiv:1401.7502).  \newblock {\em J. Reine Angew. Math.},
  2016.

\bibitem{CD:Answer_2017} F.~Catanese and M.~Dettweiler.  \newblock
  Answer to a question by {F}ujita on {V}ariation of {H}odge
  {S}tructures.  \newblock {\em Adv. Stud. in Pure Math.},
  74-04(04):73--102, 2017.

\bibitem{C-P_TheGriffiths_1995} A.~Collino and G.~P. Pirola.
  \newblock The {G}riffiths infinitesimal invariant for a curve in its
  {J}acobian.  \newblock {\em Duke Math. J.}, 78(1):59--88, 1995.

\bibitem{CF10} E.~Colombo and P.~Frediani.  \newblock Siegel metric
  and curvature of the moduli space of curves.  \newblock {\em
    Trans. Amer. Math. Soc.}, 362(3):1231--1246, 2010.

\bibitem{CFG15} E.~Colombo, P.~Frediani, and A.~Ghigi.  \newblock On
  totally geodesic submanifolds in the {J}acobian locus.  \newblock
  {\em Internat. J. Math.}, 26(1):1550005, 21, 2015.

\bibitem{cpt} E.~Colombo, G.~P. Pirola, and A.~Tortora.  \newblock
  Hodge-{G}aussian maps.  \newblock {\em Ann. Scuola Norm. Sup. Pisa
    Cl. Sci. (4)}, 30(1):125--146, 2001.

\bibitem{dejong-zhang} J.~de~Jong and S.-W. Zhang.  \newblock Generic
  abelian varieties with real multiplication are not {J}acobians.
  \newblock In {\em Diophantine geometry}, volume~4 of {\em CRM
    Series}, pages 165--172. Ed. Norm., Pisa, 2007.

  
\bibitem{MC91} M.~Coppens and G.~Martens.  \newblock Secant spaces and
  {C}lifford's theorem.  \newblock {\em Compositio Math.},
  78(2):193--212, 1991.


\bibitem{Fuj78b} T.~Fujita.  \newblock The sheaf of relative canonical
  forms of a {K}\"ahler fiber space over a curve.  \newblock {\em
    Proc. Japan Acad. Ser. A Math. Sci.}, 54(7):183--184, 1978.

\bibitem{deba} A.~Ghigi.  \newblock On some differential-geometric
  aspects of the Torelli map.  \newblock {\em arXiv:1809.06315}.
  \newblock To appear on {\em Boll. Unione. Mat. Ital.}: {\tt
    https://doi.org/10.1007/s40574-018-0171-3}.
  
\bibitem{GST17} V.~{Gonz{\'a}lez-Alonso}, L.~{Stoppino}, and
  S.~{Torelli}.  \newblock {On the rank of the flat unitary factor of
    the Hodge bundle}.  \newblock {\em ArXiv: 1709.05670}, Sept. 2017.

\bibitem{GT18} V.~{Gonz{\'a}lez-Alonso} and S.~{Torelli}.  \newblock
  {On the rank of the flat unitary factor of the Hodge bundle.}
  \newblock 2018.  \newblock {\em arXiv:1812.05891}.


\bibitem{Grif_OnthePeriodsI&II_1969} P.~A. Griffiths.  \newblock On
  the periods of certain rational integrals. {I}, {II}.  \newblock
  {\em Ann. of Math. (2) 90 (1969), 460-495; ibid. (2)}, 90:496--541,
  1969.

\bibitem{gm1} S.~Grushevsky and M.~M\"oller.  \newblock Shimura curves
  within the locus of hyperelliptic {J}acobians in genus 3.  \newblock
  {\em Int. Math. Res. Not. IMRN}, (6):1603--1639, 2016.


  
\bibitem{hain} R.~Hain.  \newblock Locally symmetric families of
  curves and {J}acobians.  \newblock In {\em Moduli of curves and
    abelian varieties}, Aspects Math., E33, pages
  91--108. Friedr. Vieweg, Braunschweig, 1999.

\bibitem{liu-yau-ecc} K.~Liu, X.~Sun, X.~Yang, and S.-T. Yau.
  \newblock Curvatures of moduli spaces of curves and applications.
  \newblock {\em Asian J. Math.} 21 (2017), no. 5, 841-854.


  
\bibitem{lu-zuo-HE} X.~Lu and K.~Zuo.  \newblock The {O}ort conjecture
  on {S}himura curves in the {T}orelli locus of hyperelliptic curves.
  \newblock {\em J. Math. Pures Appl. (9)}, 108(4):532--552, 2017.

  
\bibitem{MNP16} V.~Marcucci, J.~C. Naranjo, and G.~P. Pirola.
  \newblock Isogenies of {J}acobians.  \newblock {\em Algebr. Geom.},
  3(4):424--440, 2016.

%
  
\bibitem{moonen-oort} B.~Moonen and F.~Oort.  \newblock The {T}orelli
  locus and special subvarieties.  \newblock In {\em {H}andbook of
    {M}{oduli: Volume II}}, pages 549--94.  International {P}ress,
  Boston, MA, 2013.


\bibitem{mumford} D.~Mumford.  \newblock A note of {S}himura's paper
  ``{D}iscontinuous groups and abelian varieties''.  \newblock {\em
    Math. Ann.}, 181:345--351, 1969.

    
\bibitem{oort-steenbrink} F.~Oort and J.~Steenbrink.  \newblock The
  local {T}orelli problem for algebraic curves.  \newblock In {\em
    Journ\'ees de {G}\'eometrie {A}lg\'ebrique d'{A}ngers, {J}uillet
    1979/{A}lgebraic {G}eometry, {A}ngers, 1979}, pages 157--204.
  Sijthoff \& Noordhoff, Alphen aan den Rijn, 1980.


\bibitem{PT} G.~P. Pirola and S.~Torelli.  \newblock Massey products
  and {F}ujita decompositions.  \newblock {\em ArXiv:1710.02828},
  2017.


\bibitem{toledo} D.~Toledo.  \newblock Nonexistence of certain closed
  complex geodesics in the moduli space of curves.  \newblock {\em
    Pacific J. Math.}, 129(1):187--192, 1987.

  
\bibitem{V_HodgeTheoryI_2002} C.~Voisin.  \newblock {\em Hodge theory
    and complex algebraic geometry. {I}}, volume~76 of {\em Cambridge
    Studies in Advanced Mathematics}.  \newblock Cambridge University
  Press, Cambridge, 2002.  \newblock Translated from the French
  original by Leila Schneps.

\end{thebibliography}
\end{document}